\documentclass{amsart}
\usepackage{amsmath}
\usepackage{amsfonts}
\usepackage{amssymb}
\usepackage{color}

\newcommand{\ba}{\mathbf{a}} \newcommand{\bx}{\mathbf{x}} \newcommand{\by}{\mathbf{y}}
\newcommand{\cF}{\mathcal{F}}

\newcommand{\vac}{|0\rangle}

\newtheorem{thm}{Theorem}[section]

\newtheorem{lem}[thm]{Lemma}
\newtheorem{prop}[thm]{Proposition}

\theoremstyle{definition}

\theoremstyle{remark}
\newtheorem{rem}{Remark}[section]

\newcommand{\half}{\frac{1}{2}}

\newcommand{\be}{\begin{equation}}
\newcommand{\ee}{\end{equation}}
\newcommand{\bea}{\begin{eqnarray}}
\newcommand{\eea}{\end{eqnarray}}
\newcommand{\ben}{\begin{eqnarray*}}
\newcommand{\een}{\end{eqnarray*}}
\newcommand{\bet}{\begin{equation}
\begin{split}}
\newcommand{\eet}{\end{split}
\end{equation}}

\DeclareMathOperator{\Span}{span}

\begin{document}

\title{On Fermionic Representation of the Framed Topological Vertex}
\date{}
\subjclass[2010]{51P05}
\thanks{\emph{Key words}. the topological vertex, Bogoliubov transforms, the ADKMV Conjecture}
\author{Fusheng Deng \and Jian Zhou}
\address{Fusheng Deng: \ School of Mathematical Sciences, Graduate University of Chinese Academy of Sciences, Beijing 100049, China ; Department of Mathematical Sciences\\Tsinghua University\\Beijing, 100084, China }
\email{fshdeng@gucas.ac.cn}
\address{Jian Zhou: Department of Mathematical Sciences\\Tsinghua University\\Beijing, 100084, China}
\email{jzhou@math.tsinghua.edu.cn}

\begin{abstract}
The Gromov-Witten invariants of $\mathbb{C}^3$ with branes is encoded in the topological vertex
which has a very complicated combinatorial expression.
A simple formula for the topological vertex was proposed by Aganagic {\em et al} in the fermionic picture.
We will propose a similar formula for the framed topological vertex and
prove it in the case when there are one or two branes.
\end{abstract}

\maketitle

\section{Introduction}

The topological vertex \cite{AKMV, LLLZ} is the basic building block for
the theory of open and closed Govomov-Witten invariants of toric Calabi-Yau threefolds.
It encodes open Gromov-Witten invariants of $\mathbb{C}^3$ with three special $D$-branes,
and Gromov-Witten invariants of any toric Calabi-Yau threefold, both open and closed,
can be computed from it by  certain explicit gluing process.
One way to understand the gluing is through taking inner products or vacuum expectation values
on the bosonic Fock space \cite{Zh2}.
More precisely,
the boundary condition on each of the three $D$-branes
is indexed by a partition $\mu^i$, $i=1,2,3$.
It is well-known that the space $\Lambda$ of symmetric functions has some natural basis indexed by partitions
(e.g. the Newton functions).
One then understands the topological vertex as an element in the tensor product $\Lambda^{\otimes 3}$,
and the gluing is achieved by taking inner products on the corresponding copies of $\Lambda$.
In this picture the topological vertex has an extremely complicated combinatorial expression
in terms of skew Schur functions.
Suggested by the boson-fermion correspondence,
a deep conjecture was made in \cite{AKMV} and \cite{ADKMV}
that the topological vertex has a surprisingly simple expression in the fermionic picture:
It is a Bogoliubov transform of the fermionic vacuum,
i.e. the fermionic vacuum acted upon by an exponential of a quadratic expression of fermionic operators.
We will refer to this as the ADKMV Conjecture.
See \S \ref{sec:ADKMV} for a precise statement.

A straightforward application as pointed out in \cite{ADKMV}
is related to integrable hierarchies:
the one-legged case is related to the KP hierarchy,
the two-legged case to the $2$-dimensional Toda hierarchy,
and the three-legged case to the $3$-component KP hierarchy.
The one-legged and the two-legged cases can also be seen directly from the bosonic picture \cite{Zh3},
but the three-legged case can only be seen through the fermionic picture.

In the rest of this paper, after reviewing some preliminaries in \S 2,
we will first propose in \S 3 a generalization of the ADKMV Conjecture to the framed topological vertex.
For the precise statement see \S \ref{sec:Framed}.
We will refer to this conjecture as the Framed ADKMV Conjecture.
Secondly,
we will prove the one-legged and two-legged cases of the Framed ADKMV Conjecture
in \S \ref{sec:OneLegged} and \S \ref{sec:TwoLegged} respectively.
In the final \S 6 we will derive a determinatal formula for the framed topological vertex in the three-legged
case based on the Framed ADKMV Conjecture.

\vspace{.1in}
{\em Acknowledgements}.
The work was partially done during the first author's attending
the mathematical seminars supported by Morningside Center of CAS.
The first author is partially supported by NSFC grants
(11001148 and 10901152) and the President Fund of GUCAS. The second author
is partially supported by two NSFC grants (10425101 and 10631050) and a 973
project grant NKBRPC (2006cB805905).

\section{Preliminaries}

\subsection{Partitions}
A partition $\mu$ of a positive integral number $n$ is a decreasing finite sequence of integers $\mu_1\geq\cdots \geq\mu_l>0$,
such that $|\mu|:= \mu_1 + \cdots + \mu_l = n$.
The following number associated to $\mu$ will be useful in this paper:
\be
 \kappa_\mu = \sum_{i=1}^l \mu_i(\mu_i - 2i + 1).
\ee
It is very useful to graphically represent a partition by its Young diagram.
This leads to many natural definitions.
First of all,
by transposing the Young diagram one can define the conjugate $\mu^t$ of $\mu$.
Secondly
assume the Young diagram of $\mu$ has $k$ boxes in the diagonal.
Define $m_i = \mu_i - i$ and $n_i = \mu^t_i - i$ for $i = 1, \cdots , k$,
then it is clear that $m_1> \cdots > m_k \geq 0$ and  $n_1> \cdots > n_k \geq 0$.
The partition $\mu$ is completely determined by the numbers $m_i , n_i$.
We often denote the partition $\mu$ by $(m_1, \dots , m_k | n_1, \dots , n_k)$,
this is called the Frobenius notation.
A partition of the form $(m|n)$ in Frobenius notation is called a hook partition.

For a box $e$ at the position $(i , j)$ in the Young diagram of $\mu$,
define its content by $c(e) = j-i$.
Then it is easy to see that
\be \label{eqn:Kappa}
\kappa_\mu = 2\sum_{e\in \mu}c(e).
\ee
Indeed,
\ben
&& \sum_{e\in \mu}c(e) = \sum_{i=1}^l \sum_{j=1}^{\mu_i} (j-i)
= \sum_{i=1}^n (\half \mu_i(\mu_i+1) - i \mu_i) = \half \kappa_\mu.
\een
A straightforward application of \eqref{eqn:Kappa} is the following:

\begin{lem} \label{lm:kappa}
Let $\mu = (m_1, m_2, \dots, m_k | n_1, n_2, \dots, n_k)$ be a partition written in the Frobenius notation. Then we have
\be
\kappa_\mu = \sum_{i=1}^k m_i(m_i+1) - \sum_{i=1}^k n_i(n_i+1).
\ee
In particular,
\be
\kappa_{(m_1, m_2, \dots, m_k | n_1, n_2, \dots, n_k)}
= \sum_{i=1}^k \kappa_{(m_i|n_i)}.
\ee
\end{lem}

\begin{proof}
It is clear that:
\ben
\sum_{e \in (m_1, \dots, m_k|n_1, \dots, n_k)} c(e)
= \sum_{i=1}^k (\sum_{c= 1}^{m_i} c - \sum_{c=1}^{n_i} c)
= \sum_{i=1}^k (\half m_i(m_i+1) - \half n_i(n_i+1)).
\een
\end{proof}

\subsection{Schur functions and skew Schur functions}
Let $\Lambda$ be the space of symmetric functions in $\bx = (x_1, x_2, \dots)$.
For a partition $\mu$, let $s_\mu:=s_\mu(\bx)$ be the Schur function in $\Lambda$.
If we write $\mu = (m_1, \cdots , m_k | n_1, \cdots , n_k)$ in Frobenius notation, then there is a
determinantal formula that expresses $s_\mu$ in terms of $s_{(m|n)}$.

\begin{prop} \label{prop:SchurHook}\cite[p. 47, Example 9]{Macdonald}
Let $\mu = (m_1, \cdots , m_k | n_1, \cdots , n_k)$ be a partition in Frobenius notation, then
$$s_\mu = \det(s_{(m_i|n_j)})_{1\leq i , j \leq k }.$$
\end{prop}

The inner product on the space $\Lambda$ is defined by setting the set of Schur functions as an orthonormal basis.
Given two partitions $\mu$ and $\nu$, the skew Schur functions $s_{\mu/\nu}$ is defined by the condition
$$(s_{\mu/\nu} , s_\lambda) = (s_\mu , s_\nu s_\lambda)$$
for all partitions $\lambda$. This is equivalent to define
$$s_{\mu/\nu} = \sum_{\lambda}c_{\nu\lambda}^\mu s_\lambda,$$
where the constants $c_{\nu\lambda}^\mu$ are the structure constants
(called the Littlewood-Richardson coefficients) defined by
\be
s_\nu s_\lambda = \sum_{\gamma}c_{\nu\lambda}^\gamma s_\gamma.
\ee
If we write $\mu = (m_1, \cdots , m_k | n_1, \cdots , n_k)$ and $\nu = (s_1, \cdots , s_r | t_1, \cdots , t_r)$ in Frobenius notations,
then $s_{\mu/\nu} = 0$ unless $r\leq k$ and $s_i\leq m_i , t_i \leq n_i$ for $i = 1, \cdots , r$.
There is a determinantal formula for $s_{\mu/\nu}$ in terms of $s_{(m|n)}$ as follows:

\begin{prop} \label{prop:SkewSchurHook} \cite[p. 88, Example 22(a)]{Macdonald}
Let  $\mu = (m_1, \cdots , m_k | n_1, \cdots , n_k)$ and $\nu = (s_1, \cdots , s_r | t_1, \cdots , t_r)$ be two partitions in Frobenius notations, then
\be \label{eqn:SkewSchurHook}
s_{\mu/\nu}=(-1)^r \det\left(\begin{array}{cc}
(s_{(m_i|n_j)})_{k\times k } & (s_{(m_i-s_j-1|0)})_{k\times r}\\
(s_{(0|n_j-t_i-1)})_{r\times k}    &  0_{r\times r}
\end{array}\right).
\ee
In particular,
\be
s_{(m|n)/(s|t)} = h_{m-s} e_{n-t}.
\ee
\end{prop}

\subsection{Specialization of symmetric functions}
Let $q^{\rho}:=(q^{-1/2}, q^{-3/2}, \dots)$.
It is easy to see that
\be
p_n(q^\rho) = \frac{1}{q^{n/2} - q^{-n/2}} = \frac{1}{[n]},
\ee
where $[n] = q^{n/2} - q^{-n/2}.$  A very interesting fact is that with this specialization the Schur functions
also have very simple expressions.

\begin{prop} \label{prop:SchurSpec} \cite{Zh4}   For any partition $\mu$, one has
$$s_\mu(q^\rho) = q^{\kappa_\mu/4}\frac{1}{\prod_{e\in \mu}[h(e)]},$$
where $h(e)$ is the hook number of $e$.
\end{prop}

\subsection{Fermionic Fock space }

We say a set of half-integers $A = \{a_1, a_2, \dots \}\subset \mathbb{Z}+\frac{1}{2}$, $a_1>a_2> \cdots$, is admissible if it satisfies the following two conditions:
\begin{itemize}
\item[1.] $\mathbb{Z}_- + \frac{1}{2}\backslash A$ is finite and
\item[2.] $A\backslash \mathbb{Z}_- + \frac{1}{2}$ is finite,
\end{itemize}
where  $\mathbb{Z}_-$ is the set of negative integers.

Consider the linear space $W$ spanned by a basis $\{\underline{a}| a\in \mathbb{Z}+\frac{1}{2}\}$,
indexed by half-integers.
For an admissible set $A = \{a_1> a_2> \dots\}$,
we associate an element $\underline{A}\in \wedge^\infty W$ as follows:
$$\underline{A} = \underline{a_1}\wedge \underline{a_2} \wedge \cdots.$$
Then the free fermionic Fock space $\mathcal{F}$ is defined as
$$\cF = \Span \{\underline{A}: \; A\subset \mathbb{Z}+\frac{1}{2}\; \text{is admissible} \}.$$
One can define an inner product on $\mathcal{F}$ by taking
$\{\underline{A}:\; A\subset \mathbb{Z}+\frac{1}{2}\; \text{is admissible} \}$ as an orthonormal basis.

For $\underline{A} = \underline{a_1}\wedge \underline{a_2} \wedge \cdots
\in \mathcal{F}$,
define its charge as:
$$|A\backslash \mathbb{Z}_- + \frac{1}{2}| - |\mathbb{Z}_- + \frac{1}{2}\backslash A|.$$
Denote by $\cF^{(n)} \subset \mathcal{F}$ the subspace spanned by $\underline{A}$ of charge $n$,
then there is a decomposition
$$\mathcal{F} = \bigoplus_{n\in \mathbb{Z}} \cF^{(n)}.$$
An operator on $\mathcal{F}$ is called charge 0 if it preserves the above decomposition.

The charge 0 subspace $\cF^{(0)}$ has a basis indexed by partitions:
\be
|\mu\rangle:= \underline{\mu_1 - \frac{1}{2}} \wedge \underline{\mu_2 - \frac{3}{2}}\wedge \cdots \wedge
 \underline{\mu_l -\frac{2l-1}{2}}\wedge \underline{-\frac{2l+1}{2}}\wedge \cdots
\ee
where $\mu = (\mu_1, \cdots , \mu_l)$,
i.e.,
$|\mu\rangle = \underline{A_\mu}$, where $A_\mu =(\mu_i - i + \half)_{i=1, 2, \dots}$.
If $\mu = (m_1, \cdots , m_k | n_1, \cdots , n_k)$ in Frobenius notation, then
\begin{equation}
|\mu\rangle = \underline{m_1+\frac{1}{2}}\wedge \cdots \wedge \underline{m_k+\frac{1}{2}}\wedge \underline{-\frac{1}{2}}
\wedge \underline{-\frac{3}{2}} \wedge \cdots \wedge \widehat{\underline{-n_k-\frac{1}{2}}} \wedge \cdots \wedge
\widehat{\underline{-n_1-\frac{1}{2}}} \wedge \cdots .
\end{equation}
In particular,
when $\mu$ is the empty partition,
we get:
$$|0\rangle := \underline{-\frac{1}{2}}\wedge \underline{-\frac{3}{2}}\wedge \cdots \in \mathcal{F}.$$
It will be called the fermionic vacuum vector.

We now recall the creators and annihilators on $\mathcal{F}$.
For $r \in \mathbb{Z}+\frac{1}{2}$,
define operators $\psi_r$ and $\psi^*_r$ by
\begin{eqnarray*}
&\psi_r (\underline{A}) =
\begin{cases}
(-1)^{k}\underline{a_1}\wedge\cdots\wedge\underline{a_k}\wedge \underline{r}\wedge\underline{a_{k+1}}\wedge\cdots, & \text{if $a_k > r > a_{k+1}$ for some $k$}, \\
0, &  \text{otherwise};
\end{cases}\\
&\psi^*_r(\underline{A}) =
\begin{cases}
(-1)^{k+1}\underline{a_1}\wedge\cdots\wedge \widehat{\underline{a_k}}\wedge\cdots, & \text{if $a_k = r$ for some $k$}, \\
0, &  \text{otherwise}.
\end{cases}
\end{eqnarray*}
The anti-commutation relations for these operators are
\begin{equation} \label{eqn:CR}
[\psi_r,\psi^*_s]:= \psi_r\psi^*_s + \psi^*_s\psi_r = \delta_{r,s}id
\end{equation}
and other anti-commutation relations are zero.
It is clear that for $r > 0$,
\begin{align}
\psi_{-r} \vac & = 0, & \psi_r^* \vac & = 0,
\end{align}
so the operators $\{\psi_{-r}, \psi_r^*\}_{r > 0}$ are called the fermionic annihilators.
For a partition $\mu = (m_1, m_2, . . ., m_k | n_1, n_2, . . ., n_k)$, it is clear that
\be
|\mu\rangle = (-1)^{n_1 + n_2 + . . . + n_k}\prod_{i=1}^k \psi_{m_i+\frac{1}{2}} \psi_{-n_i-\frac{1}{2}}^*|0\rangle .
\ee
So the operators $\{\psi_{r}, \psi_{-r}^*\}_{r > 0}$ are called the fermionic creators.
The normally ordered product is defined as
\begin{equation*}
:\psi_r\psi^*_r: =
\begin{cases}
 \psi_r\psi^*_r, & r>0, \\
- \psi^*_r\psi_r, & r<0.
\end{cases}
\end{equation*}
In other words,
an annihilator is always put on the right of a creator.

\subsection{The boson-fermion correspondence}
For any integer $n$, define an operator $\alpha_n$ on the fermionic Fock space $\mathcal{F}$ as follows:
\begin{equation*}
\alpha_n = \sum_{r\in \mathbb{Z} + \frac{1}{2}}:\psi_r\psi^*_{r+n}:
\end{equation*}
Let
$\mathcal{B} = \Lambda[z , z^{-1}]$
be the bosonic Fock space, where $z$ is a formal variable.
Then the  boson-fermion correspondence is a linear isomorphism
$\Phi: \mathcal{F} \rightarrow \mathcal{B}$ given by
\begin{equation}
u\mapsto z^m \langle\underline{0}_m | e^{\sum_{n=1}^\infty \frac{p_n}{n}\alpha_n}u\rangle ,\ \ u\in \cF^{(m)}
\end{equation}
where $|\underline{0}_m\rangle = \underline{-\frac{1}{2}+m}\wedge \underline{-\frac{3}{2}+m}\wedge\cdots$.
It is clear that $\Phi$ induces an isomorphism between $\cF^{(0)}$ and $\Lambda$. Explicitly, this isomorphism is given by
\begin{equation}\label{boson-fermion}
|\mu\rangle \longleftrightarrow s_\mu.
\end{equation}

The boson-fermionic correspondence plays an important role in Kyoto school's theory
of integrable hierarchies.
For example,

\begin{prop}\label{bilinear relation tau fermion}
If $\tau\in \Lambda$ corresponds to $|v\rangle\in F^{(0)}$ under the boson-fermion correspondence, then $\tau$ is a $tau$-function of the KP
hierarchy in the Miwa variable $t_n = \frac{p_n}{n}$ if and only if $|v\rangle$ satisfies the bilinear relation
\begin{equation}\label{bilinear relation tau fermion 1}
\sum_{r\in \mathbb{Z} + \frac{1}{2}}\psi_r |v\rangle\otimes \psi^*_r |v\rangle = 0.
\end{equation}
\end{prop}

A state $|v\rangle\in \cF^{(0)}$ satisfies the bilinear relation
\eqref{bilinear relation tau fermion 1} if and only if it lies in the orbit $\widehat{GL_\infty}|0\rangle$.
There is also a multi-component generalization of the boson-fermion correspondence
which can be used to study multi-component KP hierarchies \cite{KL}.

\section{The ADKMV Conjecture and Its Framed Generalization}

\subsection{The topological vertex}
The topological vertex introduced in \cite{AKMV} is defined by
\begin{eqnarray} \label{eqn:TV}
W_{\mu^1, \mu^2, \mu^3}(q)
= \sum_{\rho^1, \rho^3}c_{\rho^1(\rho^3)^t}^{\mu^1(\mu^3)^t}q^{\kappa_{\mu^2}/2+\kappa_{\mu^3}/2}
\frac{W_{(\mu^2)^t\rho^1}(q)W_{\mu^2(\rho^3)^t}(q)}{W_{\mu^2\emptyset}(q)},
\end{eqnarray}
where
$$ c_{\rho^1(\rho^3)^t}^{\mu^1(\mu^3)^t}
= \sum_{\eta} c_{\eta\rho^1}^{\mu^1}c_{\eta(\rho^3)^t}^{(\mu^3)^t}.
$$
It can also be rewritten as follows (see e.g. \cite{Zh4}):
\be \label{eqn:WSkewSchur}
W_{\mu^1, \mu^2, \mu^3}(q)
= (-1)^{|\mu^2|} q^{\kappa_{\mu^3}/2}
s_{(\mu^2)^t}(q^{-\rho}) \sum_{\eta}
s_{\mu^1/\eta}(q^{(\mu^2)^t+\rho})
s_{(\mu^3)^t/\eta}(q^{\mu^2+\rho}).
\ee
The framed topological vertex in framing $(a_1, a_2, a_3)$ is given by:
\be
W^{(a_1, a_2, a_3)}_{\mu^1, \mu^2, \mu^3}(q) = q^{a_1\kappa_{\mu^1}/2+a_2\kappa_{\mu^2}/2+a_3\kappa_{\mu^3}/2}
W_{\mu^1, \mu^2, \mu^3}(q).
\ee
Let
\be
Z^{(a_1, a_2, a_3)}(q; \bx^1;\bx^2;\bx^3) = \sum_{\mu^{1},\mu^{2},\mu^{3}}
W^{(a_1,a_2,a_3)}_{\mu^1, \mu^2, \mu^3}(q) s_{\mu^1}(\bx^1) s_{\mu^2}(\bx^2) s_{\mu^3}(\bx^3).
\ee
Even though the topological vertex is presented here in its combinatorial expression,
its significance lies in its geometric origin as open Gromov-Witten invariants.
In the mathematical theory of the topological vertex \cite{LLLZ},
the open Gromov-Witten invariants are defined by localizations
on relative moduli spaces.
This leads to some special Hodge integrals on the Deligne-Mumford moduli spaces,
whose generating series can be shown to be $Z^{(a_1, a_2, a_3)}(q; \bx^1;\bx^2;\bx^3)$.
Closed and open Gromov-Witten invariants of local toric Calabi-Yau $3$-folds
can be obtained from the topological vertex by suitable gluing process.

\subsection{The ADMKV Conjecture}\label{sec:ADKMV}

It is conjectured in \cite{AKMV} and \cite{ADKMV} that the topological vertex
has a simple expression in the fermionic picture as follows.
On the three-component femionic Fock space $\mathcal{F} \otimes \mathcal{F} \otimes \mathcal{F}$,
define for $i=1, 2, 3$ operators $\psi^i_r$ and $\psi^{i*}_r$, $r \in {\mathbb Z} + \half$.
They act on the $i$-th factor of the tensor product
as the operators $\psi_r$ and$\psi_r^*$ respectively,
and we use the Koszul sign convention for the anti-commutation relations of these operators, i.e., we set
$$[\psi^i_r , \psi^j_s]=[\psi^i_r ,\psi^{j*}_s]=[\psi^{i*}_r , \psi^{j*}_s] =0$$
for $i\neq j$ and $r , s \in {\mathbb Z} + \half$.
Let
\be
\psi^{ij}_{mn} := \psi^i_{m+\half} \psi^{j*}_{-n-\half}.
\ee
Let $|\mu^1, \mu^2, \mu^3\rangle = |\mu^1\rangle \otimes |\mu^2\rangle \otimes |\mu^3\rangle
\in \cF^{(0)} \otimes \cF^{(0)} \otimes \cF^{(0)}$.
Then the ADKMV Conjecture states that
\be
W_{\mu^1, \mu^2, \mu^3}(q)
= \langle \mu^1, \mu^2, \mu^3| \exp \biggl( \sum_{\substack{ i,j=1,2,3\\ m,n\geq 0}}
A_{mn}^{ij}(q) \psi^{ij}_{mn} \biggr) \vac \otimes \vac \otimes \vac,
\ee
where for $i=1,2,3$,
\bea
&& A_{mn}^{ii}(q)  = (-1)^n  \frac{q^{m(m+1)/4-n(n+1)/4}}{[m+n+1][m]![n]!}, \\
&& A_{mn}^{i(i+1)}(q)  = (-1)^n q^{m(m+1)/4-n(n+1)/4+1/6}
\sum_{l=0}^{\min(m,n)} \frac{q^{(l+1)(m+n-l)/2}}{[m-l]![n-l]!}, \\
&& A_{mn}^{i(i-1)}(q)  = (-1)^{n+1} q^{-m(m+1)/4+n(n+1)/4-1/6}
\sum_{l=0}^{\min(m,n)} \frac{q^{-(l+1)(m+n-l)/2}}{[m-l]![n-l]!}.
\eea
Here it is understood that $A^{34}_{mn} = A^{31}_{mn}$ and $A^{10}_{mn} =A^{13}_{mn}$.
This is very surprising because in the bosonic picture the expression for the topological vertex is very complicated.

\subsection{The Framed ADMKV Conjecture}\label{sec:Framed}
We make the following generalization of the above ADKMV Conjecture to the framed topological vertex:
\be \label{eqn:FramedADKMV}
W^{(\ba)}_{\mu^1, \mu^2, \mu^3}(q)
= \langle \mu^1, \mu^2, \mu^3| \exp  \biggl( \sum_{\substack{ i,j=1,2,3\\ m,n\geq 0}}
A_{mn}^{ij}(q;\ba) \psi_{mn}^{ij} \biggr)  \vac \otimes \vac \otimes \vac
\ee
for $A^{ij}_{mn}(q;\ba)$ similar to $A^{ij}_{mn}(q)$ above:
\bea\label{framed coefficients 1}
&& A^{ii}_{mn}(q;\ba)
       = (-1)^n q^{(2a_i+1)(m(m+1)-n(n+1))/4}\frac{1}{[m+n+1][m]![n]!}, \\
\label{framed coefficients 2}
&& A_{mn}^{i(i+1)}(q;\ba) = (-1)^n q^{\frac{(2a_i+1)m(m+1)-(2a_{i+1}+1)n(n+1)}{4}+\frac{1}{6}}\sum_{l=0}^{\min(m , n)}
 \frac{q^{\frac{1}{2}(l+1)(m+n-l)}}{[m-l]![n-l]!},\\
 \label{framed coefficients 3}
&& A_{mn}^{i(i-1)}(q;\ba) = -(-1)^n q^{\frac{(2a_i+1)m(m+1)-(2a_{i-1}+1)n(n+1)}{4}-\frac{1}{6}}\sum_{l=0}^{\min(m , n)}
 \frac{q^{-\frac{1}{2}(l+1)(m+n-l)}}{[m-l]![n-l]!}.
\eea
Here $\ba = a_1, a_2, a_3$.
We refer to this conjecture as the Framed ADKMV Conjecture.

We derive $A^{ij}_{m,n}(q;\ba)$ by the same method as for the derivation of $A^{ij}_{mn}(q)$
in \cite[\S 5.11]{ADKMV}.
For details, see \S \ref{sec:A}.
It is surprising that there is only little difference between them.

A straightforward application of the ADKMV Conjecture and the Framed ADKMV Conjecture is that
they establish a connection between the topological vertex and integrable hierarchies as pointed out in \cite{ADKMV}.

\section{Proof of The One-Legged Case}\label{sec:OneLegged}

In this section,
as a warm up exercise
we will derive a fermionic representation of the framed one-legged topological vertex,
hence establishing the one-legged case of the Framed ADKMV Conjecture.

\subsection{The framed one-legged topological vertex in terms of Schur functions}
The generating functional of the Gromov-Witten invariants of $\mathbb{C}^3$ with one brane is
encoded in $W^{(a,0,0)}_{\mu, (0), (0)}$.
It is also the generating function of certain Hodge integrals on the moduli spaces of pointed stable curves.
Let
\be
Z^{(a)}(q; \bx) = \sum_{\mu}
W^{(a,0,0)}_{\mu, (0), (0)}(q) s_{\mu}(\bx).
\ee
By (\ref{eqn:WSkewSchur}) one then has:
\be
Z^{(a)}(q;\bx) = \sum_{\mu} q^{a \kappa_\mu /2} s_\mu (q^\rho) s_\mu(\bx).
\ee
By \eqref{boson-fermion},
this corresponds to an element $V^{(a)}(q)$ in the charge 0 ferminonic Fock subspace $\cF^{(0)}$:
\be \label{eqn:V(q,a)}
V^{(a)}(q) = \sum_{\mu}q^{a\kappa_\mu/2} s_\mu (q^\rho) |\mu\rangle.
\ee
\subsection{Proof of the one-legged case of the Framed ADKMV Conjecture}

By the Framed ADKMV Conjecture we should have
\be \label{eqn:ADKMV1}
V^{(a)}(q) = \exp(\sum_{m , n = 0}^{\infty}A_{mn}(q;a) \psi_{m+\frac{1}{2}}\psi_{-n-\frac{1}{2}}^*)|0\rangle
\ee
for some $A_{mn}(q;a)$.

\begin{lem} \label{lm:Det1}
The following identity holds:
\be
\begin{split}
& \exp(\sum_{m , n = 0}^{\infty}A_{mn}(q;a)  \psi_{m+\frac{1}{2}}\psi_{-n-\frac{1}{2}}^*)|0\rangle\\
=&\sum_{\mu=(m_1, \dots, m_k|n_1, \dots, n_k)} (-1)^{n_1 + . . . + n_k} \det(A_\mu)|\mu\rangle,
\end{split}
\ee
where $(A_\mu) = (A_{m_i n_j}(q;a))_{1\leq i , j \leq k }$ if $\mu = (m_1, m_2, . . ., m_k| n_1, n_2, . . ., n_k)$.
\end{lem}

\begin{proof}
By the commutation relations (\ref{eqn:CR}),
operators $\{\psi_{m+\frac{1}{2}}\psi_{-n-\frac{1}{2}}^*\}_{m,n\geq 0}$ commute with each other
and their squares are all the $0$-operator.
Therefore,
one has:
\ben
&& \exp(\sum_{m , n = 0}^{\infty}A_{mn}(q;a)  \psi_{m+\frac{1}{2}}\psi_{-n-\frac{1}{2}}^*)|0\rangle\\
&=&\prod_{m,n \geq 0}(1 + A_{mn}(q;a)  \psi_{m+\frac{1}{2}}\psi_{-n-\frac{1}{2}}^*)|0\rangle\\
&=&\sum_{\mu=(m_1, \dots, m_k|n_1, \dots, n_k)} \sum_{\sigma \in S_k} sign(\sigma)
\prod_{i=1}^k A_{m_in_{\sigma(i)}}(q;a) \cdot \prod_{i=1}^k (\psi_{m_k+\half}\psi^*_{-n_k-\half})
|0\rangle \\
&=&\sum_{\mu=(m_1, \dots, m_k|n_1, \dots, n_k)} (-1)^{n_1 + . . . + n_k} \det(A_\mu)|\mu\rangle.
\een
\end{proof}

Set $\mu = (m|n)$. If we assume (\ref{eqn:ADKMV1}),
we must have
\be
(-1)^nA_{mn}(q;a) = \langle (m|n)|V(q;a) = q^{a \kappa_{(m|n)} /2} s_{(m|n)} (q^\rho) .
\ee
By Lemma \ref{lm:Det1}, Lemma \ref{lm:kappa} and Proposition \ref{prop:SchurSpec},
\begin{equation}\label{Amn}
\begin{split}
A_{mn}(q;a) &= (-1)^n q^{a \kappa_{(m|n)}/2}s_{(m|n)}(q^\rho) \\
       &= (-1)^n q^{\frac{(m-n)(m+n+1)(2a+1)}{4}}\frac{1}{[m+n+1][m]![n]!}.
\end{split}
\end{equation}

\begin{thm}
In the case of one-legged topological vertex,
the Framed ADKMV Conjecture holds for the above $A_{mn}(q;a)$.
\end{thm}

\begin{proof}
For $\mu = (m_1, m_2, . . ., m_k | n_1, n_2, . . ., n_k)$, by Proposition \ref{prop:SchurHook},
Lemma \ref{lm:kappa} and \eqref{Amn}, we get
\ben
q^{a\kappa_\mu /2}s_\mu(q^\rho)
& = & q^{\sum_{i=1}^{k}a m_i(m_i+1) /2- \sum_{j=1}^{k}a n_j(n_j+1) /2} \cdot \det(s_{(m_i|n_j)}(q^\rho))_{i,j=1, \dots, k}\\
& = & q^{\sum_{i=1}^k am_i(m_i+1)/2}  \cdot \det(q^{-an_j(n_j+1)/2} \cdot s_{(m_i|n_j)}(q^\rho))_{i,j=1, \dots, k} \\
& = & \det(q^{am_i(m_i+1)/2} q^{-an_j(n_j+1)/2} s_{(m_i|n_j)}(q^\rho))_{i,j=1, \dots, k} \\
& = &(-1)^{n_1 + . . . + n_k} \det(A_{m_in_j})_{i,j=1, \dots, k}.
\een
The proof is completed by Lemma \ref{lm:Det1} and \eqref{eqn:ADKMV1}.
\end{proof}

For later reference,
note we have proved the following identity:
\be \label{eqn:Va}
V^{(a)}(q) =  \sum_{\mu=(m_1, \dots, m_k|n_1, \dots, n_k)}
\det \big(q^{a\kappa_{(m_i|n_j)}/2} s_{(m_i|n_j)}(q^\rho)\big)_{1 \leq i, j \leq k}|\mu\rangle.
\ee

\section{Proof of The Two-Legged Case}\label{sec:TwoLegged}

\subsection{The framed two-legged topological vertex in terms of skew Schur functions}
The framed two-legged topological vertex
encodes the open Gromov-Witten invariants of  $\mathbb{C}^3$ with two branes:
\be \label{eqn:Z2Legged}
Z^{(a_1, a_2)}(q; \bx; \by)
= \sum_{\mu^1, \mu^2} W_{\mu^1, \mu^2, (0)}^{(a_1, a_2,0)}(q) s_{\mu^1}(\bx) s_{\mu^2}(\by).
\ee
Recall the following identity proved in \cite{Zh4}:
\be
W_{\mu^1, \mu^2, (0)}(q)
= q^{\kappa_{\mu^2}/2} W_{\mu^1, (\mu^2)^t}(q).
\ee
The following identity proved in \cite{Zh1} will play a key role below:
\be \label{eqn:Wmunu}
W_{\mu, \nu}(q)
= (-1)^{|\mu|+|\nu|}
q^{\frac{\kappa_{\mu}+\kappa_{\nu}}{2}}
\sum_{\eta} s_{\mu/\eta}(q^{-\rho})s_{\nu/\eta}(q^{-\rho}).
\ee
Based on this formula,
the following formula is proved in \cite{Zh4}:
\be
W_{(\mu^1)^t, (\mu^2)^t}(q^{-1}) = (-1)^{|\mu^1|+|\mu^2|} W_{\mu^1, \mu^2}(q). \label{eqn:Wtt}
\ee
Therefore, \eqref{eqn:Z2Legged} can be rewritten as follows:
\be\label{eqn:2Legged}
Z^{(a_1,a_2)}(q;\bx;\by)
= \sum_{\mu, \nu}\left(q^{\frac{(a_1+1)\kappa_\mu+a_2\kappa_\nu}{2}}
\sum_{\eta}s_{\mu^t/\eta}(q^\rho)s_{\nu/\eta}(q^\rho)\right )s_\mu(\bx) s_\nu(\by).
\ee
By the boson-fermion correspondence \eqref{boson-fermion},
this corresponds to the following element in the femionic picture:
\be \label{eqn:W2Legged}
V^{(a_1,a_2)}(q) = \sum_{\mu, \nu}\left(q^{\frac{(a_1+1)\kappa_\mu+a_2\kappa_\nu}{2}}\sum_{\eta}s_{\mu^t/\eta}(q^\rho)s_{\nu/\eta}(q^\rho)\right ) |\mu\rangle\otimes |\nu\rangle.
\ee
Using this we will prove the two-legged case of the Framed ADKMV Conjecture.

\begin{thm} \label{thm:Main}
There is an operator
$$T(q;a_1,a_2) = \exp(\sum_{i , j = 1, 2}\sum_{m , n = 0}^{\infty} A^{ij}_{mn}(q;a_1,a_2)\psi^{i}_{m+\frac{1}{2}}\psi^{j*}_{-n-\frac{1}{2}})$$
where the coefficients $A_{mn}^{ij}(q;a_1,a_2)$,  for $ m , n \geq 0$ and $i ,\ j = 1 , 2$,
are given by \eqref{eqn:Aii}, \eqref{eqn:A12} and \eqref{eqn:A21} below,
such that
\be
W_{\mu, \nu, (0)}^{(a_1,a_2,0)}(q) = \langle\mu, \nu|T(q;a_1,a_2) \vac \otimes \vac.
\ee
\end{thm}

\subsection{The determination of $A^{ij}_{mn}(q; a_1, a_2)$} \label{sec:A}
Note that the charge 0 subspace $(\cF\otimes\cF)^{(0)}$ of $\cF\otimes\cF$ has a decomposition
$$(\cF\otimes\cF)^{(0)} = \bigoplus_{n\in\mathbb{Z}}\cF^{(n)}\otimes\cF^{(-n)}.$$
The Framed ADKMV Conjecture predicts the existence of an operator $T$ of the form
\be
T(q;a_1,a_2) = \exp(\sum_{i , j = 1, 2}\sum_{m , n = 0}^{\infty}
A^{ij}_{mn}(q;a_1,a_2) \psi^{ij}_{mn}),
\ee
such that $V^{(a_1,a_2)}(q)$ is the projection of
the element $T(q;a_1,a_2) (|0\rangle\otimes |0\rangle)\in (\cF\otimes\cF)^{(0)}$ onto $\cF^{(0)}\otimes \cF^{(0)}$.
In this subsection we modify the method in \cite[\S 5.11]{ADKMV}
to the framed case to derive explicit expressions for $A^{ij}_{mn}(q; a_1, a_2)$.

Because the operators $\{\psi^{ij}_{mn}\}$ commute with each other and square to zero,
we have
\be
\begin{split}
T(q;a_1,a_2) = \prod_{m,n \geq 0} & [(1 + A^{11}_{mn}(q,a_1,a_2)\psi_{mn}^{11} )(1 + A^{22}_{mn}(q,a_1,a_2)\psi_{mn}^{22} )\\
& \cdot (1 + A^{12}_{mn}(q,a_1,a_2)\psi_{mn}^{12} )(1 + A^{21}_{mn}(q,a_1,a_2)\psi_{mn}^{21} )].
\end{split}
\ee

Take $\mu = (m|n)$ and $\nu = \emptyset$ or take $\nu = (m|n)$ and $\mu = \emptyset$,
as in the one-legged case we get for $i=1,2$:
\begin{equation} \label{eqn:Aii}
\begin{split}
A^{ii}_{mn}(q;a_1,a_2) &= (-1)^n q^{a_i\kappa_{(m|n)}/2}s_{(m|n)}(q^\rho) \\
       &= (-1)^n q^{(2a_i+1)(m(m+1)-n(n+1))/4}\frac{1}{[m+n+1][m]![n]!}.
\end{split}
\end{equation}

Take $\mu = (m|n)$ and $\nu = (m'|n')$,
then it is clear that the coefficient of $|(m|n)\rangle\otimes |(m'|n')\rangle$ in
$T(q;a_1,a_2) (|0\rangle\otimes |0\rangle)$ is
$$(-1)^{n+n'}(A^{11}_{mn}(q;a_1,a_2) A^{22}_{m'n'}(q;a_1,a_2)
-  A^{12}_{mn'}(q,a_1,a_2) A^{21}_{m'n}(q,a_1,a_2)).$$
Assuming the Framed ADKMV Conjecture,
one should have:
\be \label{eqn:A2}
\begin{split}
& q^{(a_1+1)\kappa_{(m|n)}/2+a_2\kappa_{(m'|n')}/2}\sum_{\eta}s_{(n|m)/\eta}(q^\rho)s_{(m'|n')/\eta}(q^\rho)\\
= &(-1)^{n+n'}(A^{11}_{mn}(q,a_1,a_2) A^{22}_{m'n'}(q,a_1,a_2)
-  A^{12}_{mn'}(q,a_1,a_2) A^{21}_{m'n}(q,a_1,a_2) ).
\end{split}
\ee
The left-hand side  can be rewritten as follows:
\ben
&&q^{\frac{(a_1+1)\kappa_{(m|n)}+a_2\kappa_{(m'|n')}}{2}}\sum_{\eta}s_{(n|m)/\eta}(q^\rho)s_{(m'|n')/\eta}(q^\rho)\\
&=& q^{\frac{(a_1+1)\kappa_{(m|n)}+a_2\kappa_{(m'|n')}}{2}}(s_{(n|m)}(q^\rho)s_{(m'|n')}(q^\rho) + \sum_{\eta\neq\emptyset}s_{(n|m)/\eta}(q^\rho)s_{(m'|n')/\eta}(q^\rho)).
\een
Therefore,
by \eqref{eqn:Aii} we have
\be \label{eqn:A12A21}
\begin{split}
& A_{mn'}^{12}(q;a_1,a_2) \cdot A_{m'n}^{21}(q;a_1,a_2) \\
= & (-1)^{n+n'+1}q^{\frac{(a_1+1)\kappa_{(m|n)}+a_2\kappa_{(m'|n')}}{2}} \sum_{\eta\neq\emptyset}s_{(n|m)/\eta}(q^\rho)s_{(m'|n')/\eta}(q^\rho).
\end{split}
\ee
By Proposition \ref{prop:SkewSchurHook}, we have
\begin{equation}
\begin{split}
s_{(n|m)/(s|t)}(q^\rho)& = s_{(m-s)}(q^\rho)s_{(1^{n-t})}(q^\rho)\\
                         & = q^{(m-s)(m-s-1)/4-(n-t)(n-t-1)/4} \frac{1}{[m-s]![n-t]!}.
\end{split}
\end{equation}
We will take $A_{00}^{12}(q;a_1,a_2)= q^{1/6}$, $A_{00}^{21}(q;a_1,a_2) = -q^{-1/6}$
as in \cite{ADKMV}.
If we set $m' = n =0$ and $m = n' = 0$ in \eqref{eqn:A12A21}
respectively,
we get:
\bea
&& A_{mn}^{12}(q;a_1,a_2) = (-1)^n q^{\frac{(2a_1+1)m(m+1)-(2a_2+1)n(n+1)}{4}+\frac{1}{6}}\sum_{s=0}^{\min(m , n)}
 \frac{q^{\frac{1}{2}(s+1)(m+n-s)}}{[m-s]![n-s]!}, \label{eqn:A12} \\
&& A_{mn}^{21}(q;a_1,a_2) = -(-1)^n q^{\frac{(2a_2+1)m(m+1)-(2a_1+1)n(n+1)}{4}-\frac{1}{6}}\sum_{s=0}^{\min(m , n)}
 \frac{q^{-\frac{1}{2}(s+1)(m+n-s)}}{[m-s]![n-s]!} \label{eqn:A21}
\eea

\subsection{Some technical lemmas}
For simplicity of notations,
we will write $A^{ij}_{mn} : = A^{ij}_{mn}(q;a_1, a_2)$.
For a partition $\mu = (m_1, m_2, \cdots, m_k | n_1, n_2, \cdots, n_k)$ in Frobenius notation and $i, j = 1 ,2$, we
define $A^{ij}_{\mu}$ to be the matrix $(A^{ij}_{m_an_b})_{k\times k}$.

For a set $N=\{n_1, \dots, n_k\}$ of numbers, let $||N||$ be the sum of the numbers in $N$.
I.e.,
\be
||N||= \sum_{i=1}^k n_i.
\ee
For simplicity of notations we will write
$f(q^\rho)$ as $\bar{f}$ for $f \in \Lambda$,
e.g.,
$s_{(m|n)/\eta}(q^\rho)$ will be written as $\bar{s}_{(m|n)/\eta}$.

\begin{lem} \label{lm:DetAii}
Suppose that $(M|N) = (m_1, \dots, m_k|n_1, \dots, n_k)$ is a partition in Frobenius notation.
Then we have for $l=1,2$,
\be
\det A^{ll}_{(M|N)} = (-1)^{||N||} q^{a_l \kappa_{(M|N)}/2}
\cdot\det (\bar{s}_{(M|N)}),
\ee
where $(\bar{s}_{(M|N)}) = (\bar{s}_{(m_i|n_j)})_{1 \leq i, j \leq k}$.
\end{lem}

\begin{proof}
One can use Lemma \ref{lm:kappa} to get:
\be
\kappa_{(m|n)} = \kappa_{(m|0)} + \kappa_{(0|n)}.
\ee
By \eqref{eqn:Aii},
\ben
\det A^{ll}_{(M|N)}
& = & \det ( (-1)^{n_j} q^{a_l\kappa_{(m_i|0)}/2} \cdot q^{a_l \kappa_{(0|n_j)}/2} \cdot \bar{s}_{(m_i|n_j)} )_{1 \leq i, j \leq k} \\
& = &  (-1)^{\sum_{j=1}^k n_j} q^{a_l \sum_{i=1}^k \kappa_{(m_i|0)}/2}
q^{\sum_{j=1}^k a_l\kappa_{(0|n_j)}/2} \det ( \bar{s}_{(m_i|n_j)} )_{1 \leq i, j \leq k} \\
& = & (-1)^{||N||} q^{a_l \kappa_{(M|N)}/2}\cdot\det (\bar{s}_{(M|N)}).
\een
\end{proof}

\begin{lem} \label{lm:DetA12A21}
Given $r \geq 1$,
suppose $A=\{a_1 > \dots > a_r\}$,  $A'=\{a'_1 > \dots > a'_r\}$,
$B=\{b_1 > \dots > b_r\}$, $B'=\{b'_1 > \dots > b'_r\}$.
one has
\begin{multline}
\det(A^{12}_{(A|B')})\det(A^{21}_{(A'|B)})
= (-1)^{||B||+||B'||+ r} q^{(a_1+1)\kappa_{(A|B)}/2 + a_2 \kappa_{(A'|B')}/2}\\
\cdot \sum_{\substack{s_1> \cdots  >s_r\\t_1> \cdots  > t_r}} \det(\bar{e}_{a_j -t_i} )
\cdot \det (\bar{h}_{b_j - s_i} ) \cdot \det(\bar{h}_{a'_j-s_i} ) \cdot \det ( \bar{e}_{b'_j-t_i} ).
\end{multline}
\end{lem}
\begin{proof}
Expanding the determinants one has

\ben
&&  \det(A^{12}_{(A|B')})\det(A^{21}_{(A'|B)})
=    \sum_{\sigma,\tau \in S_r}\epsilon(\sigma\tau)A^{12}_{a_ib'_{\sigma(i)}}A^{21}_{a'_i b_{\tau(i)}}\\
& =& \sum_{\sigma,\tau \in S_r}\epsilon(\sigma\tau) \prod_{i=1}^r
\biggl( (-1)^{b_{\tau(i)}+b'_{\sigma(i)} + 1 }
q^{\frac{(a_1+1)\kappa_{(a_i|b_{\tau(i)})}+a_2\kappa_{(a_i'|b'_{\sigma(i)})} }{2}} \\
&& \cdot \sum_{\eta_i \neq \emptyset}
   \bar{s}_{(b_{\tau(i)}|a_i)/\eta_i} \bar{s}_{(a'_i|b'_{\sigma(i)})/\eta_i} \biggr),
\een
where $\epsilon(\sigma\tau)$ is the sign of the permutation $\sigma\tau$, and in the second equality we have used \eqref{eqn:A12A21}.

By Proposition \ref{prop:SkewSchurHook},

\ben
&&  \det(A^{12}_{(A|B')})\det(A^{21}_{(A'|B)}) \\
&=&\sum_{\sigma,\tau \in S_r}\epsilon(\sigma\tau) \prod_{i=1}^r\biggl( (-1)^{b_{\tau(i)}+b'_{\sigma(i)} +1 }
q^{\frac{(a_1+1)\kappa_{(a_i|b_{\tau(i)})}+a_2\kappa_{(a_i'|b'_{\sigma(i)})} }{2}} \\
&& \cdot \sum_{(s_i|t_i)} \bar{h}_{b_{\tau(i)}-s_i} \bar{e}_{a_i-t_i}
\bar{h}_{a'_i-s_i} \bar{e}_{b'_{\sigma(i)}-t_i} \biggr) \\
&=&\sum_{s_i , t_i \geq 0, i=1, \dots, r}\sum_{\sigma,\tau \in S_r}\epsilon(\sigma\tau)
\prod_{i=1}^r \biggl( (-1)^{b_{\tau(i)}+b'_{\sigma(i)} + 1 }
q^{\frac{(a_1+1)\kappa_{(a_i|0)} + (a_1+1)\kappa_{(0|b_{\tau(i)})} }{2}} \\
&& \cdot q^{\frac{a_2\kappa_{(a_i'|0)}+a_2\kappa_{(0|b'_{\sigma(i)})} }{2}}
\bar{h}_{b_{\tau(i)}-s_i} \bar{e}_{a_i-t_i} \bar{h}_{a'_i-s_i} \bar{e}_{b'_{\sigma(i)}-t_i} \biggr)\\
& = & (-1)^r \sum_{s_i , t_i \geq 0, i=1, \dots, r}
\prod_{i=1}^r  ((-1)^{b_i+b_i'}q^{\frac{(a_1+1)\kappa_{(a_i|b_i)} + a_2 \kappa_{(a'_i|b_i')} }{2}}
\bar{e}_{a_i-t_i} \bar{h}_{a'_i-s_i}) \\
&& \cdot \det (   \bar{h}_{b_j-s_i})_{1 \leq i,j \leq r} \cdot
\det (  \bar{e}_{b'_j-t_i} )_{1 \leq i, j \leq r} \\
& = & (-1)^{||B||+||B'||+ r} q^{(a_1+1)\kappa_{(A|B)}/2 + a_2 \kappa_{(A'|B')}/2}
\sum_{s_i , t_i \geq 0, i=1, \dots, r}
\prod_{i=1}^r  ( \bar{e}_{a_i-t_i} \bar{h}_{a'_i-s_i}) \\
&& \cdot \det (   \bar{h}_{b_j-s_i})_{1 \leq i,j \leq r} \cdot
\det (  \bar{e}_{b'_j-t_i} )_{1 \leq i, j \leq r} .
\een
Now note $\det (h_{b_{j}-s_i})_{1 \leq i,j \leq r} = 0$
if $s_i = s_j$ for some $1 \leq i < j \leq r$,
and $\det ( e_{b'_{j}-t_i} )_{1 \leq i, j \leq r} = 0$
if $t_i = t_j$ for some $1 \leq i < j \leq r$.
Therefore,

\ben
&& \sum_{s_i , t_i \geq 0, i=1, \dots, r}
\prod_{i=1}^r (\bar{e}_{a_i-t_i}\bar{h}_{a'_i-s_i})\cdot \det (\bar{h}_{b_j-s_i})_{1 \leq i,j \leq r} \cdot
\det ( \bar{e}_{b'_j-t_i} )_{1 \leq i, j \leq r} \\
&=&\sum_{\substack{s_1, \cdots  , s_r\ are\ distinct\\t_1,\cdots , t_r\ are\ distinct}}
\prod_{i=1}^r (\bar{e}_{a_i-t_i}\bar{h}_{a'_i-s_i})\cdot \det (\bar{h}_{b_j-s_i})_{1 \leq i,j \leq r} \cdot
\det ( \bar{e}_{b'_j-t_i} )_{1 \leq i, j \leq r} \\
&=&\sum_{\substack{s_1, \cdots  , s_r\ are\ distinct\\t_1,\cdots , t_r\ are\ distinct}}
\sum_{\sigma,\tau \in S_r}\epsilon(\sigma\tau) \prod_{i=1}^r
 \bar{h}_{b_{\tau(i)}-s_i}\bar{e}_{a_i-t_i}\bar{h}_{a'_i-s_i}\bar{e}_{b'_{\sigma(i)}-t_i}
\een
\ben
&=& \sum_{\substack{s_1> \cdots  >s_r\\t_1> \cdots  > t_r}}\sum_{x,y,x',y' \in S_r} \epsilon(xyx'y')
   \prod_{i=1}^r \bar{e}_{a_{x(i)} - t_{i}}\bar{h}_{b_{y(i)}-s_i}\bar{h}_{a'_{x'(i)} - s_{i}}\bar{e}_{b'_{y'(i)}-t_i} \\
&=&\sum_{\substack{s_1> \cdots  >s_r\\t_1> \cdots  > t_r}} \det(\bar{e}_{a_j -t_i} )
\cdot \det (\bar{h}_{b_j - s_i} ) \cdot \det(\bar{h}_{a'_j-s_i} ) \cdot \det ( \bar{e}_{b'_j-t_i} ).
\een
The proof is complete.
\end{proof}

\subsection{From fermionic representation to determinantal representation}
 For a partition $\mu = (m_1, m_2, \cdots, m_k | n_1, n_2, \cdots, n_k)$ in Frobenius notation and $i, j = 1 ,2$, we
define an operator
\begin{equation}
\psi^{ij}_{\mu} = \prod_{a=1}^k \psi^{ij}_{m_an_a}.
\end{equation}
By \eqref{eqn:CR}, we can expand $T$ as follows:
\ben
&& \prod_{m,n} (1 + A^{11}_{mn}\psi_{mn}^{11} )(1 + A^{22}_{mn}\psi_{mn}^{22} )(1 + A^{12}_{mn}\psi_{mn}^{12} )(1 + A^{21}_{mn}\psi_{mn}^{21} )\\
  & = &(1\ +\sum_{\substack {m_1>\cdots>m_k\\ n_1>\cdots>n_k}} \det(A^{11}_{m_in_j}) \prod_{i=1}^k \psi^{11}_{m_in_i} )
    \cdot (1\ +\sum_{\substack {m_1>\cdots>m_k\\ n_1>\cdots>n_k}} \det(A^{22}_{m_in_j}) \prod_{i=1}^k \psi^{22}_{m_in_i} )\\
    && \cdot (1\ +\sum_{\substack {m_1>\cdots>m_k\\ n_1>\cdots>n_k}}\det(A^{12}_{m_in_j}) \prod_{i=1}^k \psi^{12}_{m_in_i} )
     \cdot (1\ +\sum_{\substack {m_1<\cdots<m_k\\ n_1<\cdots<n_k}}\det(A^{21}_{m_in_j}) \prod_{i=1}^k \psi^{21}_{m_in_i} )\\
  &=&\sum_{\mu^1}\det A_{\mu^1}^{11}\psi^{11}_{\mu^1} \cdot \sum_{\mu^2}\det A_{\mu^2}^{22}\psi^{22}_{\mu^1}
    \cdot \sum_{\mu^3}\det A_{\mu^3}^{12}\psi^{12}_{\mu^3} \cdot \sum_{\mu^4}\det A_{\mu^4}^{21}\psi^{21}_{\mu^4}\\
  &=&\sum_{\mu^1 , \mu^2,\mu^3 ,\mu^4}\det A_{\mu^1}^{11}\det A_{\mu^2}^{22}\det A_{\mu^3}^{12}\det A_{\mu^4}^{21}
   \cdot \psi^{11}_{\mu^1}\psi^{22}_{\mu^2}\psi^{12}_{\mu^3}\psi^{21}_{\mu^4},
\een
where the summation is over all partitions  $\mu^1 , \mu^2 ,\mu^3 ,\mu^4$, including the empty partition, and we set $det A_{\mu}^{ij}= 1$,   and
$\psi^{ij}_{\mu} = 1$ is $\mu$ is the empty partition.
Now let $\mu = (M|N)=(m_1, m_2, \cdots, m_k | n_1, n_2, \cdots, n_k)$ and $\nu = (M'|N')=(m_1', m_2', \cdots, m_l' | n'_1, n'_2,\cdots, n'_l)$
be two partitions.
Denote by $C^{(a_1,a_2)}_{\mu\nu}(q)$    the inner product of $|\mu\rangle\otimes |\nu\rangle$
with $T(q;a_1,a_2)(|0\rangle\otimes |0\rangle)$.

We need some notations.
For a partition $\mu = (m_1, m_2, \cdots, m_k | n_1, n_2, \cdots, n_k)$,
define $r(\mu) = k$ to be the length of the diagonal of its Young diagram.

Let $M$ be a set of nonnegative integers $\{m_1> m_2 > \cdots > m_k\}$ written in decreasing order.
For a subset $A = \{m_{i_1}> \cdots>m_{i_r}\}$ of $M$,
also written in decreasing order,
denote by $\varepsilon(M/A)$ the sign of the permutation
$$(m_{i_1}, \cdots ,m_{i_r}, m_{j_1}, \cdots , m_{j_{k-r}})\rightarrow (m_1, \cdots , m_k),$$
where $(m_{j_1}, \cdots m_{j_{k-r}})$ is the set $M\backslash A$ written in decreasing order.
Let $\mu = (M|N)$ and $\gamma = (A|B)$ be two partitions,
we define $\gamma < \mu$ if $A\subset M$ and $B\subset N$.
If $\gamma < \mu$ holds, then $\mu\backslash \gamma := (M\backslash A | N\backslash B)$ is naturally defined as a partition.
By \eqref{eqn:CR}, the following Lemma is easy to prove.

\begin{lem} \label{lm:Fermion}
Let $\mu = (M|N)$ and $\gamma = (A|B)$ be two partitions such that $\gamma < \mu$, then
\begin{equation}
\psi_{\gamma}\psi_{\mu\backslash \gamma} = \varepsilon(M/A)\varepsilon(N/B)\psi_{\mu}\ \ .
\end{equation}
\end{lem}

Now it is straightforward to get the following

\begin{lem}\label{lem:fermion-determinant}
Let $C^{(a_1,a_2)}_{\mu\nu}(q) = \langle \mu, \nu|T(q;a_1,a_2) \vac \otimes \vac$.
Then one has
\begin{equation*}
\begin{split}
 C^{(a_1,a_2)}_{\mu\nu}(q) =&(-1)^{||N||+||N'||} \sum_{\substack{\gamma=(A|B)<\mu \\ \gamma'=(A'|B')<\nu \\
            r(\gamma)=r(\gamma')}}(-1)^{r(\gamma)}\varepsilon(M/A)\varepsilon(N/B)
            \varepsilon(M'/A')\varepsilon(N'/B')\\
& \cdot            \det(A^{11}_{(M\backslash A|N\backslash B)}) \det(A^{22}_{(M'\backslash A'|N'\backslash B')})
            \det(A^{12}_{(A|B')}) \det(A^{21}_{(A'|B)}).
\end{split}
\end{equation*}
\end{lem}

\begin{prop}\label{2-part det}
We have
\begin{equation} \label{eqn:Sum}
\begin{split}
&C^{(a_1,a_2)}_{\mu\nu}(q)\\
 =&\sum_{r=0}^{\min(k,l)} \sum_{*r } \varepsilon(M/A)\varepsilon(N/B)
   \varepsilon(M'/A')\varepsilon(N'/B')\\
  &\cdot q^{(a_1\kappa_{(M|N)} + a_2 \kappa_{(M'|N')} + \kappa_{(A|B)})/2}\cdot   \det(\bar{s}_{(M\backslash A|N \backslash B)}) \cdot \det(\bar{s}_{(M'\backslash A'|N' \backslash B')})\\
 &\cdot \sum_{\substack{s_1> \cdots  >s_r\\t_1> \cdots  > t_r}} \det(\bar{e}_{a_j -t_i} )
\cdot \det (h_{b_j - s_i} ) \cdot \det(h_{a'_j-s_i} ) \cdot \det ( \bar{e}_{b'_j-t_i} ),
\end{split}
\end{equation}
where the condition ($*r$) in the summation is given by
\begin{equation}\label{summation condition}
\gamma=(A|B)<\mu ,\ \gamma'=(A'|B')<\nu, \ r(\gamma)=r(\gamma')=r.
\end{equation}
\end{prop}

\begin{proof}
For $r\geq 0$, let
\begin{equation}  \label{eqn:Cr}
\begin{split}
& C^{(a_1,a_2),r}_{\mu\nu}(q)  =(-1)^{||N||+||N'||}  \sum_{*r}(-1)^{r}\varepsilon(M/A)\varepsilon(N/B)
  \varepsilon(M'/A')\varepsilon(N'/B') \\
 & \qquad \cdot \det(A^{11}_{(M\backslash A|N\backslash B)}) \det(A^{22}_{(M'\backslash A'|N'\backslash B')})
            \det(A^{12}_{(A|B')})\det(A^{21}_{(A'|B)}),
\end{split}
\end{equation}
where the condition ($*r$) in the summation is given by \eqref{summation condition}.
Then we have
\begin{equation}
C^{(a_1,a_2)}_{\mu\nu}(q) = \sum_{r=0}^{\min(k,l)}C^{(a_1,a_2),r}_{\mu\nu}(q)  .
\end{equation}
For $r=0$, by \eqref{eqn:Cr}, Proposition \ref{prop:SchurHook} and \eqref{eqn:Aii}, we have
\ben
C^{(a_1,a_2),0}_{\mu\nu}(q) &= &(-1)^{||N|| + ||N'||} \det(A^{11}_{(M|N)})\det(A^{22}_{(M'|N')})\\
& = & q^{\frac{a_1\kappa_{(M|N)} + a_2\kappa_{(M'|N')}}{2}}\det(\bar{s}_{(m_i|n_j)}(q^\rho))\det(\bar{s}_{(m_i'|n_j')}(q^\rho))\\
& = &  q^{\frac{a_1\kappa_{(M|N)} + a_2\kappa_{(M'|N')}}{2}}\bar{s}_{(M|N)}(q^\rho)\bar{s}_{(M'|N')}(q^\rho).
\een
For $r>0$, we use Lemma \ref{lm:DetAii} , Lemma \ref{lm:DetA12A21} and Lemma \ref{lem:fermion-determinant} to get:
\ben
&&C^{(a_1,a_2),r}_{\mu\nu}(q)\\
&=&\sum_{*r}\varepsilon(M/A)\varepsilon(N/B)\varepsilon(M'/A')\varepsilon(N'/B')\\
&& \cdot (-1)^{||N\backslash B||} \det(A^{11}_{(M\backslash A|N\backslash B)}) \cdot   (-1)^{||N'\backslash B'||}\det(A^{22}_{(M'\backslash A'|N'\backslash B')})\\
&&\cdot (-1)^{||B||+||B'||+ r}\det(A^{12}_{(A|B')}) \det(A^{21}_{(A'|B)})\\
&=&  q^{\frac{a_1\kappa_{(M\backslash A|N \backslash B)}}{2}}
 \det(\bar{s}_{(M\backslash A|N \backslash B)})q^{\frac{a_2\kappa_{(M'\backslash A'|N' \backslash B')}}{2}}
 \det(\bar{s}_{(M'\backslash A'|N' \backslash B')})\\
&& \cdot  q^{(a_1+1)\kappa_{(A|B)}/2 + a_2 \kappa_{(A'|B')}/2}\\
&&\sum_{\substack{s_1> \cdots  >s_r\\t_1> \cdots  > t_r}} \det(\bar{e}_{a_j -t_i} )
\cdot \det (\bar{h}_{b_j - s_i} ) \cdot \det(\bar{h}_{a'_j-s_i} ) \cdot \det ( \bar{e}_{b'_j-t_i} ) \\
& = &\sum_{*r} \varepsilon(M/A)\varepsilon(N/B)
   \varepsilon(M'/A')\varepsilon(N'/B')\\
 && \cdot   q^{(a_1\kappa_{(M|N)} + a_2 \kappa_{(M'|N')} + \kappa_{(A|B)})/2}
 \det(\bar{s}_{(M\backslash A|N \backslash B)}) \cdot \det(\bar{s}_{(M'\backslash A'|N' \backslash B')})\\
&& \cdot \sum_{\substack{s_1> \cdots  >s_r\\t_1> \cdots  > t_r}} \det(\bar{e}_{a_j -t_i} )
\cdot \det (\bar{h}_{b_j - s_i} ) \cdot \det(\bar{h}_{a'_j-s_i} ) \cdot \det ( \bar{e}_{b'_j-t_i} ).
\een
where the condition ($*r$) in the summation is given by \eqref{summation condition}.
\end{proof}

\subsection{Proof of the Framed ADKMV Conjecture in the two-legged case}

In this subsection we finish the proof of Theorem \ref{thm:Main}.

We now simplify the summation in \eqref{eqn:Sum}.
Let $\eta=(s_1, \dots, s_r|t_1, \dots, t_r)$.
We first take $\sum_{\gamma = (A|B), \, r(\gamma) = r}$:
\ben
&& \sum_{\substack{\gamma=(A|B)<\mu \\ r(\gamma)=r}}q^{\kappa_{(A|B)}/2}
 \varepsilon(M/A)\varepsilon(N/B)
 \det(\bar{s}_{(M\backslash A|N \backslash B)}) \\
&& \cdot  \det(\bar{e}_{a_i-t_j})_{r\times r} \cdot   \det(\bar{h}_{b_j-s_i})_{r\times r} \\
& = & \sum_{\substack{\gamma=(A|B)<\mu \\ r(\gamma)=r}}\varepsilon(M/A)\varepsilon(N/B)
   \det(\bar{s}_{(M\backslash A|N \backslash B)}) \\
&& \cdot \det(q^{\kappa_{(a_i|0)}/2} \bar{e}_{a_i-t_j})_{r\times r}
\cdot   \det(q^{\kappa_{(0|b_j)}/2} \bar{h}_{b_j-s_i})_{r\times r} \\
& = & (-1)^r \det\left(\begin{array}{cc}
(\bar{s}_{(m_i|n_j)})_{k\times k } & (q^{\kappa_{(m_i|0)}/2} \bar{e}_{m_i-t_j})_{k\times r}\\
(q^{\kappa_{(0|n_j)}/2} \bar{h}_{n_j-s_i})_{r\times k}    &  0_{r\times r}
\end{array}\right).
\een
By Proposition \ref{prop:SchurSpec} and the fact that $\kappa_{\mu^t} = - \kappa_\mu$,
we have \cite{Zh4}:
\be \label{eqn:SchurT}
s_\mu(q^\rho) = q^{\kappa_\mu/2} s_{\mu^t}(q^\rho).
\ee
It follows that
\ben
 s_{(m_i|n_j)}(q^\rho)=q^{\kappa_{(m_i|n_j)}/2}s_{(n_j|m_i)}(q^\rho).
\een
Using this we get:
\ben
&& \sum_{\substack{\gamma=(A|B)<\mu \\ r(\gamma)=r}}q^{\kappa_{(A|B)}/2}
 \varepsilon(M/A)\varepsilon(N/B)
 \det(\bar{s}_{(M\backslash A|N \backslash B)}) \\
&& \cdot  \det(\bar{e}_{a_i-t_j})_{r\times r} \cdot   \det(\bar{h}_{b_j-s_i})_{r\times r} \\
& = & (-1)^r \det\left(\begin{array}{cc}
(q^{\kappa_{(m_i|n_j)}/2} \bar{s}_{(n_j|m_i)})_{k\times k } & (q^{\kappa_{(m_i|0)}/2}\bar{e}_{m_i-t_j})_{k\times r}\\
(q^{\kappa_{(0|n_j)}/2}\bar{h}_{n_j-s_i})_{r\times k}    &  0_{r\times r}
\end{array}\right)\\
& = & q^{\kappa_\mu/2}\bar{s}_{\mu^t/\eta}(q^\rho).
\een
Similarly,
we have
\ben
&& \sum_{\substack{\gamma'=(A'|B')<\nu \\ r(\gamma')=r}} \varepsilon(M'/A')\varepsilon(N'/B')
   \det(\bar{s}_{(M'\backslash A'|N'\backslash B')}) \\
&& \cdot   \det (\bar{h}_{a'_i - s_j})_{r\times r} \cdot \det (\bar{e}_{b'_i-t_j})_{r \times r} \\
&= & (-1)^r \det\left(\begin{array}{cc}
(\bar{s}_{(m'_i|n'_j)})_{l\times l } & (\bar{h}_{m'_i-s_j})_{l\times r} \\
(\bar{e}_{n'_j-t_i})_{r\times l}    &  0_{r\times r}
\end{array}\right) \\
& = & s_{\nu/\eta}(q^\rho).
\een
Therefore,
we get:
\be
C^{(a_1,a_2)}_{\mu\nu}(q) =  q^{\frac{(a_1+1)\kappa_\mu + a_2\kappa_\nu}{2}} \sum_\eta s_{\mu^t/\eta}(q^\rho)s_{\nu/\eta}(q^\rho).
\ee
This matches with \eqref{eqn:W2Legged},
so the proof of Theorem \ref{thm:Main} is completed.

\begin{rem}
Note the charge 0 subspace $(\cF\otimes \cF)^{(0)}$ has a direct sum decomposition
\begin{equation}
(\cF\otimes \cF)^{(0)} = \bigoplus_{n\in \mathbb{Z}}(\cF^{(n)}\otimes \cF^{(-n)}).
\end{equation}
The two-legged topological vertex corresponds to only the component
of $T(|0\rangle\otimes|0\rangle)$ in $\cF^{(0)}\otimes \cF^{(0)}$.
It is interesting to find the geometric meaning of other components.
\end{rem}

From the above proof one can also see that
\begin{multline} \label{eqn:Va1a2}
V^{(a_1,a_2)}(q) = \sum_{\mu=(M|N)} \sum_{\nu=(M'|N')} \sum_{\substack{\eta=(S|T)\\ \eta < \mu, \eta < \nu \\ r(\eta) =r\geq 0}}
 q^{\frac{(a_1+1)\kappa_\mu + a_2\kappa_\nu}{2}} \\ \cdot
 \det\left(\begin{array}{cc}
( \bar{s}_{(m_i|n_j)})_{k\times k } & (\bar{e}_{m_i-t_j})_{k\times r}\\
( \bar{h}_{n_j-s_i})_{r\times k}    &  0_{r\times r}
\end{array}\right) \cdot
\det\left(\begin{array}{cc}
(\bar{s}_{(m'_i|n'_j)})_{l\times l } & (\bar{h}_{m'_i-s_j})_{l\times r} \\
(\bar{e}_{n'_j-t_i})_{r\times l}    &  0_{r\times r}
\end{array}\right).
\end{multline}

\section{Towards a Proof of The Three-Legged Case}

In this section we present an intermediate result which should be useful for a proof
of the three-legged case of the Framed ADKMV Conjecture.

\subsection{From fermionic representation to determinantal representation}

If one assumes the Framed ADKMV Conjecture,
one can determine $A_{mn}^{ij}$ ($i,j=1,2,3$) by modifying the method of \cite{ADKMV} as in \S \ref{sec:A}.
They are indeed given by \eqref{framed coefficients 1}, \eqref{framed coefficients 2},
\eqref{framed coefficients 3}.
By \eqref{eqn:CR}, we can expand $T$ as follows:
\ben
&& \prod_{i,j=1,2,3} \prod_{m,n \geq 0}  (1 + A^{ij}_{mn}\psi_{mn}^{ij} ) \\
& = & \prod_{i,j=1,2,3} (1\ +\sum_{\substack {m_1>\cdots>m_k \geq 0\\ n_1>\cdots>n_k \geq o}} \det(A^{ij}_{m_an_b})
\prod_{a=1}^k \psi^{ij}_{m_an_a} ) \\
&=& \prod_{i,j=1,2,3} \sum_{\mu^{ij}} \det (A_{\mu^{ij}}^{ij}) \psi^{ij}_{\mu^{ij}} \\
&=& \sum_{\mu^{ij}} \prod_{i,j=1,2,3} \det (A_{\mu^{ij}}^{ij})
   \cdot \prod_{i,j=1,2,3} \psi^{ij}_{\mu^{ij}},
\een
where the summation is over all partitions $\mu^{11} , \mu^{12} ,\dots ,\mu^{33}$.

Now let $\mu^i = (M^i|N^i)=(m^i_1, m^i_2, \cdots, m^i_{k_i} | n^i_1, n^i_2, \cdots, n^i_{k^i})$
(when $k_i = 0$, $\mu^i$ is the empty partition).
Denote by $C^{(\ba)}_{\mu^1, \mu^2, \mu^3}$ the right-hand side of \eqref{eqn:FramedADKMV}.
It is clear that
\be
C^{(\ba)}_{\mu^1, \mu^2, \mu^3} =
\sum_{\prod_{i,j=1,2,3} \psi^{ij}_{\mu^{ij}}
= \pm \prod_{i=1,2,3} \psi^{ii}_{\mu^i}} \pm \prod_{i,j=1,2,3} \det A_{\mu^{ij}}^{ij}.
\ee
The $\pm$ signs can be tracked off using the Koszul sign convention.
More precisely we have the following

\begin{lem}
Let $C^{(\ba)}_{\mu^1,\mu^2,\mu^3}$ be  the right-hand side of \eqref{eqn:FramedADKMV}.
Then one has
\begin{equation}\label{eqn: 3-part span}
\begin{split}
&C^{(\ba)}_{\mu^1,\mu^2, \mu^3} =\\
& (-1)^{||N^\lambda|| + ||N^\mu|| + ||N^\nu||}\sum (-1)^{r^{32}r^{12}+r^{31}r^{32}+r^{21}r^{21}+r^{32}r^{13}} \\
& \cdot \prod_{i=1}^3 \biggl( \epsilon(M^{ii}, M^{i c(i)}, M^{i c^2(i)})
\epsilon(N^{ii}, N^{c(i) i }, N^{c^2(i) i }) \biggr) \cdot \prod_{i,j =1}^3 \det(A^{ij}_{(M^{ij}|N^{ij})}).
\end{split}
\end{equation}
Here $c \in S_3$ is the $3$-cycle translation that transforms $1$ to $2$, $2$ to $3$ and $3$ to $1$,
the summation is over all partitions $\gamma^{ij}= (M^{ij}|N^{ij})$ satisfying the following conditions:
\begin{equation}
\begin{split}
& M^{ii}\amalg M^{i c(i)}\amalg  M^{i c^2(i)} = M^{i},\ i = 1, 2, 3 \\
& N^{ii}\amalg N^{c(i) i }\amalg N^{c^2(i) i } = N^{i}, \ i = 1, 2, 3 \\
& \# M^{ij} = \# N^{ij} = r^{ij}\geq 0,\ i, j = 1, 2, 3,
\end{split}
\end{equation}
and $\epsilon(M^{ii}, M^{i c(i)}, M^{i c^2(i)})$ is the sign of the transformation that rearranges
the ordered set of numbers $(M^{ii}, M^{i c(i)}, M^{i c^2(i)})$ in a decreasing order.
\end{lem}

Similar to the proof of Proposition \ref{2-part det}, one can prove the following Proposition,
which gives the determinantal form of $C^{(\ba)}_{\mu^1,\mu^2, \mu^3}$.

\begin{prop}
We have
\begin{equation} \label{eqn:Ca1a2a3}
\begin{split}
& C^{(\ba)}_{\mu^1,\mu^2, \mu^3}\\
=&\sum(-1)^{r^{21}r^{23}+r^{31}r^{32}+r^{21}r^{21}+r^{32}r^{13}+r^{13}+r^{21}+r^{32}}\\
& q^{\frac{1}{6}(r^{12}+r^{23}+r^{31}-r^{21}-r^{32}-r^{13})}q^{\sum_{i=1}^3\sum_{t=1}^{k_i}(a_i+1)(m^i_t(m^i_t+1)- n^i_t(n^i_t+1))/2}\\
& \cdot \det\left(\begin{array}{ccc}
(\bar{s}_{(m^1_i|n^1_j)})_{k_1\times k_1 } & (\bar{h}_{m^1_i-s'_j})_{k_1\times r^{13}} & (\bar{e}_{m^1_i-t_j})_{k_1\times r^{12}} \\
(\bar{h}_{n^1_j-s_i})_{r^{21}\times k_1} & 0_{r^{21}\times r^{13}}    & 0_{r^{21}\times r^{12}} \\
(\bar{e}_{n^1_j-t'_i})_{r^{31}\times k_1} & 0_{r^{31}\times r^{13}}   & 0_{r^{31}\times r^{12}}
\end{array}\right)\\
& \cdot \det\left(\begin{array}{ccc}
(\bar{s}_{(m^2_i|n^2_j)})_{k_2\times k_2 }  & (\bar{h}_{m^2_i-s_j})_{k_2\times r^{21}}& (\bar{e}_{m^2_i-t''_j})_{k_2\times r^{23}}\\
(\bar{h}_{n^2_j-s''_i})_{r^{32}\times k_2}   & 0_{r^{32}\times r^{21}} & 0_{r^{32}\times r^{23}}  \\
(\bar{e}_{n^2_j-t_i})_{r^{12}\times k_2} & 0_{r^{12}\times r^{21}} & 0_{r^{12}\times r^{23}}
\end{array}\right)\\
& \cdot \det\left(\begin{array}{ccc}
(\bar{s}_{(m^3_i|n^3_j)})_{k_3\times k_3 } & (\bar{h}_{m^1_i-s''_j})_{k_3\times r^{32}} & (\bar{e}_{m^3_i-t'_j})_{k_3\times r^{31}} \\
(\bar{h}_{n^3_j-s'_i})_{r^{13}\times k_3} & 0_{r^{13}\times r^{32}}    & 0_{r^{13}\times r^{31}} \\
(\bar{e}_{n^3_j-t''_i})_{r^{23}\times k_3} & 0_{r^{23}\times r^{32}}   & 0_{r^{23}\times r^{31}}
\end{array}\right).
\end{split}
\end{equation}
Here the summation is taken over all $r^{ij}\geq 0, i \neq j , i , j = 1, 2, 3$ satisfying the conditions
\begin{equation*}
r^{ic(i)} + r^{ic^2(i)} = r^{c(i)i} + r^{c^2(i)i}\leq k_i,\ i = 1, 2, 3 \\
\end{equation*}
 and all decreasing sequences $\{s_i\}, \{s'_i\}, \{s''_i\}, \{t_i\}, \{t'_i\}, \{t''_i\}$ of nonnegative integers.
\end{prop}

\begin{rem}
Equation \eqref{eqn:Ca1a2a3} generalizes \eqref{eqn:Va} and \eqref{eqn:Va1a2}.
\end{rem}

\maketitle

\end{document}